\documentclass{amsart}%
\usepackage{amsfonts}
\usepackage{amsmath}
\usepackage{amssymb}
\usepackage{graphicx}%
\providecommand{\U}[1]{\protect\rule{.1in}{.1in}}
\newtheorem{theorem}{Theorem}
\theoremstyle{plain}

\newtheorem{corollary}{Corollary}

\newtheorem{definition}{Definition}
\newtheorem{example}{Example}

\newtheorem{lemma}{Lemma}

\newtheorem{proposition}{Proposition}
\newtheorem{remark}{Remark}

\numberwithin{equation}{section}
\begin{document}
\title[Semi $n$-submodules]{Semi $n$-submodules of modules over commutative rings}
\author{Hani Khashan}
\address{Department of Mathematics, Faculty of Science, Al al-Bayt University, Al
Mafraq, Jordan.}
\email{hakhashan@aabu.edu.jo.}
\author{Ece Yetkin Celikel}
\address{Department of Software Engineering, Faculty of Engineering, Hasan Kalyoncu
University, Gaziantep, Turkey.}
\email{ece.celikel@hku.edu.tr, yetkinece@gmail.com.}
\thanks{This paper is in final form and no version of it will be submitted for
publication elsewhere.}

\begin{abstract}
Let $R$ be a commutative ring with identity and $M$ a unitary $R$-module. The
purpose of this paper is to introduce the concept of semi-$n$-submodules as an
extension of semi $n$-ideals and $n$-submodules. A proper submodule $N$ of $M$
is called a semi $n$-submodule if whenever $r\in R$, $m\in M$ with $r^{2}m\in
N$, $r\notin\sqrt{0}$ and $Ann_{R}(m)=0$, then $rm\in N$. Several properties,
characterizations of this class of submodules with many supporting examples
are presented. Furthermore, semi $n$-submodules of amalgamated modules are investigated.

\end{abstract}
\maketitle

\section{Introduction}

Throughout this paper, unless otherwise stated, $R$ is a commutative ring with
identity and $M$ is a unital $R$-module. Let $N$ be a submodule of an
$R$-module $M$ and $I$ be an ideal of $R$. By $Z(R)$, $reg(R),$ $\sqrt{0}$,
$Z(M)$, and $rad(N),$ we denote the set of zero-divisors of $R$, the set of
regular elements in $R$, the nil-radical of $R$, the set of all zero divisors
on $M$; i.e. $\{r\in R:rm=0$ for some $0\neq m\in M\}$ and the intersection of
all prime submodules of $M$ containing $N$, respectively. The residual $N$ by
$M$ is defined as the set $(N:_{R}M)=\{r\in R:rM\subseteq N\}$ which is an
ideal of $R$. In particular, for $m\in M$, we denote the ideals $(0:_{R}%
M)~$and $(0:_{R}m)$ by $Ann_{R}(M)$ and $Ann_{R}(m)$, respectively. The
residual $N$ by $I$ is the set $(N:_{M}I)=\{m\in M:Im\subseteq N\}$ which is a
submodule of $M$ containing $N$. More generally, for any subset $S\subseteq
R$, $(N:_{M}S)$ is a submodule of $M$ containing $N$.

The concept of prime submodules, which is an important subject in module
theory, has been widely studied by various authors. Recall that a proper
submodule $N$ of an $R$-module $M$ is a prime (resp. primary) submodule if for
$r\in R$ and $m\in M$ whenever $rm$ $\in N$, then $r\in(N:_{R}M)$ (resp.
$r\in\sqrt{(N:_{R}M)}$ or $m\in N$.\ For the sake of completeness we give some
definitions which will be used in the sequel. In \cite{Sar}, generalizing
prime submodules, the concept of semiprime submodules is first introduced. A
proper submodule $N$ of $M$ is called a semiprime submodule if for $r\in R$
and $m\in M$ whenever $r^{2}m\in N$, then $rm\in N.$ On the other hand, in
2015, R. Mohamadian \cite{Moh} introduced the concept of $r$-ideals of
commutative rings. A proper ideal $I$ of a ring $R$ is called an $r$-ideal if
whenever $a,$ $b\in R$ such that $ab\in I$ and $Ann_{R}(a)=0$, then $b\in I$
where $Ann_{R}(a)=\left\{  b\in R:ab=0\right\}  $. Afterwards, in 2017, Tekir,
Koc and Oral \cite{Tek} introduced the concept of $n$-ideals as a special kind
of $r$-ideals by considering the set of nilpotent elements instead of zero
divisors. Recently, in \cite{HE3} and \cite{HE4}, Khashan and Celikel
generalized $n$-ideal and $r$-ideals by defining and studying the classes of
semi $n$-ideals and semi $r$-ideals. A proper ideal $I$ of $R$ is called a
semi $n$-ideal (resp. semi $r$-ideal) if for $a\in R$, $a^{2}\in I$ and
$a\notin\sqrt{0}$ (resp. $Ann_{R}(a)=0$) imply $a\in I$. Later, some other
generalizations of $n$-ideals and $r$-ideals have been introduced, see for
example, \cite{Hani}, \cite{Haniece}, \cite{Haniece2} and \cite{Ece}.

In module theory, various extensions of these concepts have been studied. For
example, a proper submodule $N$ of $M$ is called an $r$-submodule (resp.
$n$-submodule) if whenever $rm\in N$ and $Ann_{M}(r)=0_{M}$ (resp.
$r\notin\sqrt{Ann_{R}(M)})$, then $m\in N$ \cite{Suat} (resp. \cite{Tek}). As
a generalization of $r$-submodules, semi $r$-submodules are introduced in
\cite{HE4}. A proper submodule $N$ of $M$ is called a semi $r$-submodule if
whenever $r\in R$, $m\in M$ with $r^{2}m\in N$, $Ann_{M}(r)=0_{M}$ and
$Ann_{R}(m)=0$, then $rm\in N$.

The aim of the paper is to introduce semi $n$-submodules as an extension of
both of semi $n$-ideals and $n$-submodules. We give many properties,
characterizations, and examples of this class of submodules. Among many
results in this paper, in Section 2, we start by giving some examples to
illustrate the place of this class of submodules in the literature (see
Example \ref{ex5})$.$ Then we study several characterizations of semi
$n$-submodules (see Theorem \ref{char1}, Theorem \ref{char}, Corollary
\ref{corr} and Corollary \ref{(N:M)}). We investigate the behavior of semi
$n$-submodules under homomorphisms, localizations, and finite Cartesian
product (see Proposition \ref{fsub}, Theorem \ref{SM} and Theorem \ref{cart}).
We conclude this section by clarifying the relation between semi
$n$-submodules of an $R$-module $M$ and the semi $n$-ideals in the
idealization ring $R(+)M$ of $M$ (see Theorem \ref{Ide}).

Let $f:$ $R_{1}\rightarrow R_{2}$ be a ring homomorphism, $J$ be an ideal of
$R_{2}$, $M_{1}$ be an $R_{1}$-module, $M_{2}$ be an $R_{2}$-module and
$\varphi:M_{1}\rightarrow M_{2}$ be an $R_{1}$-module homomorphism. The
subring
\[
R_{1}\Join^{f}J=\left\{  (r,f(r)+j):r\in R_{1}\text{, }j\in J\right\}
\]
of $R_{1}\times R_{2}$ is called the amalgamation of $R_{1}$ and $R_{2}$ along
$J$ with respect to $f$. The amalgamation of $M_{1}$ and $M_{2}$ along $J$
with respect to $\varphi$ is defined as%

\[
M_{1}\Join^{\varphi}JM_{2}=\left\{  (m_{1},\varphi(m_{1})+m_{2}):m_{1}\in
M_{1}\text{ and }m_{2}\in JM_{2}\right\}
\]
which is an $(R_{1}\Join^{f}J)$-module. In Section 3, we determine when are
some kinds of submodules of $M_{1}\Join^{\varphi}JM_{2}$ $n$-submodules and
semi $n$-submodules.

\section{Properties of Semi $n$-submodules}

In this section, among other results concerning the general properties of semi
$n$-submodules, some characterizations of this notion will be investigated.
Moreover, the relations among semi $n$-submodules and some other types of
submodules will be clarified. First, we present the fundamental definition of
semi $n$-submodules which will be studied in this paper.

\begin{definition}
Let $M$ be an $R$-module and $N$ a proper submodule of $M$. We call $N$ a semi
$n$-submodule if whenever $r\in R$, $m\in M$ with $r^{2}m\in N$, $r\notin
\sqrt{0}$ and $Ann_{R}(m)=0$, then $rm\in N$.
\end{definition}

We can easily observe that semi $n$-submodules of an $R$-module $R$ are the
same as semi $n$-ideals of $R$. Moreover, clearly the zero submodule is always
a semi $n$-submodule of $M.$ Since for $0\neq r\in R$, $Ann_{M}(r)=0_{M}$
implies $r\notin\sqrt{0}$, then any semi $n$-submodule of $M$ is a semi
$r$-submodule. In the following diagram, we illustrate the relations between
semi $n$-submodules and some other types of submodules.$\bigskip$

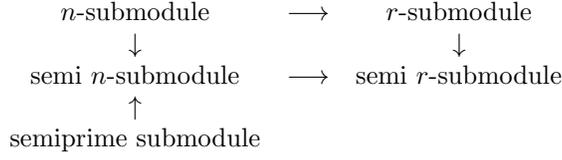
\begin{center}
$%
\begin{array}
[c]{ccc}%
n\text{-submodule} & \longrightarrow & r\text{-submodule}\\
\downarrow &  & \downarrow\\
\text{semi }n\text{-submodule} & \longrightarrow & \text{semi }%
r\text{-submodule}\\
\uparrow &  & \\
\text{semiprime submodule} &  &
\end{array}
\bigskip$
\end{center}

In the following examples, we show that the arrows in the above diagram are irreversible.

\begin{example}
\label{ex5}
\end{example}

\begin{enumerate}
\item By \cite[Example 1]{Suat}, for $k\geq2$, any proper submodule of the $%
\mathbb{Z}
$-module $%
\mathbb{Z}
_{k}$ is an $r$-submodule. Moreover, by definition, every proper submodule of
$%
\mathbb{Z}
_{k}$ is also a semi $n$-submodule. On the other hand, if $k$ is not a power
of a prime, then $%
\mathbb{Z}
_{k}$ has no $n$-submodules. Indeed, suppose say, $k=p_{1}^{m_{1}}p_{2}%
^{m_{2}}$ where $p_{1}$ and $p_{2}$ are distinct integers and $m_{1},m_{2}%
\geq1$. Let $N=\left\langle \bar{p}_{1}^{t_{1}}\bar{p}_{2}^{t_{2}%
}\right\rangle $ be a proper submodule of $%
\mathbb{Z}
_{k}$. If, say, $t_{1}=0$, then $p_{2}^{t_{2}}.\bar{1}\in N$ with
$p_{2}^{t_{2}}\notin\sqrt{Ann_{%
\mathbb{Z}
}(%
\mathbb{Z}
_{k})}=\left\langle p_{1}p_{2}\right\rangle $ and $\bar{1}\notin N$. If
$t_{1}\neq0$ and $t_{2}\neq0$, then $p_{1}^{t_{1}}.\bar{p}_{2}^{t_{2}}$ $\in
N$ with $p_{1}^{t_{1}}\notin\sqrt{Ann_{%
\mathbb{Z}
}(%
\mathbb{Z}
_{k})}$ and $\bar{p}_{2}^{t_{2}}\notin N$. Therefore, $N$ is not an
$n$-submodule of $%
\mathbb{Z}
_{k}$.

\item For a prime integer $p$, consider the $%
\mathbb{Z}
$-module
\[
M=\left\{  \frac{r}{p^{t}}+%
\mathbb{Z}
:r\in%
\mathbb{Z}
\text{, }t\in%
\mathbb{N}
\cup\left\{  0\right\}  \right\}
\]
Then any nonzero proper submodule of $M$ is of the form%
\[
N_{t_{0}}=\left\{  \frac{r}{p^{t_{0}}}+%
\mathbb{Z}
:r\in%
\mathbb{Z}
\right\}
\]
where $t_{0}\in%
\mathbb{N}
\cup\left\{  0\right\}  $, \cite{Sharp}. It is shown in \cite[Example 2]{Suat}
that any proper submodule of $M$ is an $r$-submodule. However, we show that
$N_{t_{0}}$ is never $n$-submodule for all $t_{0}\in%
\mathbb{N}
\cup\left\{  0\right\}  $. Indeed, we note that $\sqrt{Ann_{%
\mathbb{Z}
}(M)}=\left\{  0\right\}  $ since if $a\in\sqrt{Ann_{%
\mathbb{Z}
}(M)}$, then $a^{m}(\frac{1}{1}+0)=a^{m}=0$ for some $m\in%
\mathbb{N}
$ and so $a=0$. Now, for all $t_{0}\in%
\mathbb{N}
\cup\left\{  0\right\}  $, we have $p.(\frac{1}{p^{t_{0}+1}})\in N_{t_{0}}$
but $p\notin\sqrt{Ann_{%
\mathbb{Z}
}(M)}$ and $\frac{1}{p^{t_{0}+1}}\notin N_{t_{0}}$.

\item Consider the $%
\mathbb{Z}
$-module $M=%
\mathbb{Z}
_{8}\times%
\mathbb{Z}
$. Then the submodule $N=\left\langle \bar{0}\right\rangle \times\left\langle
4\right\rangle $ is a semi $r$-submodule of $M$ that is not semi
$n$-submodule. Indeed, let $r\in%
\mathbb{Z}
$ and $m=(m_{1},m_{2})\in M$ such that $r^{2}\cdot m\in N$, $Ann_{M}(r)=0_{M}$
and $Ann_{%
\mathbb{Z}
}(m)=0$. Then $r^{2}\cdot m_{1}=\bar{0}$, $r^{2}\cdot m_{2}\in\left\langle
4\right\rangle $, $m_{2}\neq0$ and $\gcd(r,8)=1$. Since $\bar{0}$ is a primary
submodule of the $%
\mathbb{Z}
$-module $%
\mathbb{Z}
_{8}$ and $r^{2}\notin\sqrt{0}$, then $m_{1}=\bar{0}$. Also, since
$\left\langle 4\right\rangle $ is a primary ideal of $%
\mathbb{Z}
$ and $r^{2}\notin\sqrt{\left\langle 4\right\rangle }$, then $m_{2}%
\in\left\langle 4\right\rangle $. It follows that $(m_{1},m_{2})\in N$ and $N$
is a semi $r$-submodule of $M$. On the other hand, we have $2^{2}\cdot(\bar
{0},1)\in N$, $2\notin\sqrt{0}$ and $Ann_{%
\mathbb{Z}
}(\bar{0},1)=0$ but $2.(\bar{0},1)\notin N$ and so $N$ is not a semi
$n$-submodule of $M$.\ 
\end{enumerate}

As a first result, we give the following characterizations of semi $n$-submodules.

\begin{theorem}
\label{char1}Let $M$ be an $R$-module and $N$ a proper submodule of $M$. Then
the following statements are equivalent.
\end{theorem}

\begin{enumerate}
\item $N$ is a semi $n$-submodule of $M$.

\item Whenever $r\in R$, $m\in M$, $k\in%
\mathbb{N}
$ with $r^{k}m\in N$, $r\notin\sqrt{0}$ and $Ann_{R}(m)=0$, then $rm\in N$.

\item For all $m\in M$, $\sqrt{(N:_{R}m)}=\sqrt{0}\cup(N:_{R}m)$ whenever
$Ann_{R}(m)=0.$
\end{enumerate}

\begin{proof}
(1)$\Rightarrow$(2) Suppose $r^{k}m\in N$, $r\notin\sqrt{0}$ and
$Ann_{R}(m)=0$ for $r\in R$, $m\in M$ and $k\in%
\mathbb{N}
$. We use the mathematical induction on $k$. If $k\leq2$, then the claim is
clear. We now assume that the result is true for all $2\lneqq t\lneqq k$ and
show that it is also true for $k.$ Suppose $k$ is even, say, $k=2l$ for some
positive integer $l.$ Since $r^{k}m=(r^{l})^{2}m\in N$ and clearly
$r^{l}\notin\sqrt{0}$, then $r^{l}m\in N$ as $N$ is a semi $n$-submodule of
$M$. By the induction hypothesis, we conclude that $rm\in N$ as needed.
Suppose $k$ is odd, so that $k+1=2s$ for some $s\lneqq k$. Then similarly, we
have $\left(  r^{s}\right)  ^{2}m\in N$ and $r^{s}\notin\sqrt{0}$ which imply
that $r^{s}m\in N$ and again by the induction hypothesis, we conclude $rm\in
N$.

(2)$\Rightarrow$(3) Let $m\in M$ such that $Ann_{R}(m)=0$. Let $r\in
\sqrt{(N:_{R}m)}$ so that $r^{k}m\in N$ for some positive integer $k.$ If
$r\notin\sqrt{0}$, then by our assumption (2), we have $rm\in N$, and so
$r\in(N:_{R}m).$ Therefore, $r\in\sqrt{0}\cup(N:_{R}m)$ and $\sqrt{(N:_{R}%
m)}\subseteq\sqrt{0}\cup(N:_{R}m)$. The reverse inclusion is clear and so the
equality holds.

(3)$\Rightarrow$(1) Let $r\in R$, $m\in M$ with $r^{2}m\in N$, $r\notin
\sqrt{0}$ and $Ann_{R}(m)=0$. As $r\in\sqrt{(N:_{R}m)}=\sqrt{0}\cup(N:_{R}m)$,
we have clearly $r\in(N:_{R}m)$ and $rm\in N$, as needed.
\end{proof}

Let $M$ be an $R$-module. Recall that an element $m\in M$ is said to be
torsion if there exists a nonzero $r\in R$ such that $rm=0$ and the set of
torsion elements of $M$ is denoted by $T(M).$ Also, recall that $M$ is called
torsion (resp. torsion-free) if $T(M)=M$ (resp. $T(M)=\{0\}$). Moreover, it is
clear that any torsion-free module is faithful. One can observe that a proper
submodule $N$ of a torsion-free $R$-module $M$ is semi $n$-submodule if and
only if $(N:_{M}r^{2})=(N:_{M}r)$ for all non-nilpotent $r\in R$.

Next, we give a further characterization for semi $n$-submodules over integral domains:

\begin{theorem}
\label{char}Let $R$ be a ring and $N$ be a proper submodule of an $R$-module
$M$. If $N$ is a semi $n$-submodule of $M$, then for $r\in R$ and a submodule
$K$ of $M$, $r^{2}K\subseteq N$, $r\notin\sqrt{0}$ and $T(K)=\{0_{M}\}$ imply
$rK\subseteq N$. Moreover, the converse holds if $R$ is an integral domain.
\end{theorem}

\begin{proof}
Suppose that $N$ is a semi $n$-submodule of $M$. Assume for $r\in R$ and a
submodule $K$ of $M$, we have $r^{2}K\subseteq N$, $r\notin\sqrt{0}$ and
$T(K)=\{0_{M}\}$. Let $0_{M}\neq k\in K.$ Then, $r^{2}k\in N$ and clearly
$Ann_{R}(k)=\{0_{R}\}.$ Since $N$ is semi $n$-submodule, we have $rk\in N$ for
all $k\in K$ and so $rK\subseteq N.$ Conversely, suppose $R$ is an integral
domain and let $r\in R$, $m\in M$ with $r^{2}m\in N$, $r\notin\sqrt{0}$ and
$Ann_{R}(m)=0$. If we put $K=Rm$, then $r^{2}K\subseteq N$ and $T(K)=\{0_{M}%
\}.$ Indeed, let $r^{\prime}m\in T(K)$ and choose $0\neq s\in R$ such that
$sr^{\prime}m=0_{M}$. As $Ann_{R}(m)=0$, we get $sr^{\prime}=0$, and so
$r^{\prime}\in Z(R)=\{0\}.$ Thus, $r^{\prime}m=0_{M}$. By assumption, we
conclude $rm\in rK\subseteq N,$ as needed.
\end{proof}

\begin{corollary}
\label{corr}Let $M$ be a torsion-free $R$-module and $N$ be a proper submodule
of $M.$\ Then the following statements are equivalent.
\end{corollary}

\begin{enumerate}
\item $N$ is a semi $r$-submodule of $M.$

\item $N$ is a semiprime submodule of $M.$

\item $N$ is a semi $n$-submodule of $M.$
\end{enumerate}

\begin{proof}
(1)$\Rightarrow$(2) Follows by \cite[Proposition 7]{HE4}.

(2)$\Rightarrow$(3) and (3)$\Rightarrow$(1) are clear from the above diagram.
\end{proof}

\begin{corollary}
\label{(N:M)2}Let $R$ be a ring and $M$ be a torsion-free $R$-module. If $N$
is a semi $n$-submodule of $M$, then $(N:_{R}M)$ is a semi $n$-ideal of $R.$
\end{corollary}

\begin{proof}
Suppose that $N$ is a semi $n$-submodule of $M$. Note that clearly,
$(N:_{R}M)$ is proper in $R$. Let $r\in R$ such that $r^{2}\in(N:_{R}M)$ and
$r\notin\sqrt{0}.$ Then $r^{2}M\subseteq N$ and $T(M)=\left\{  0_{M}\right\}
$ imply $rM\subseteq N$ by Theorem \ref{char}. Thus, $r\in(N:_{R}M)$.
\end{proof}

Recall that an $R$-module $M$ is called a multiplication module if every
submodule $N$ of $M$ has the form $IM$ for some ideal $I$ of $R$. In this
case, we have $N=(N:_{R}M)M$. Now, to prove the converse part of Corollary
\ref{(N:M)2} in finitely generated multiplication modules, we need to state
the following two lemmas.

\begin{lemma}
\label{Smith}\cite{Smith} Let $N$ be a submodule of a finitely generated
faithful multiplication $R$-module $M.$ For an ideal $I$ of $R$,
$(IN:_{R}M)=I(N:_{R}M)$, and in particular, $(IM:_{R}M)=I$.
\end{lemma}

\begin{lemma}
\cite{Majed}\label{Majed} Let $N$ be a submodule of a faithful multiplication
$R$-module $M$. If $I$ is a finitely generated faithful multiplication ideal
of $R$, then $N=(IN:_{M}I)$.
\end{lemma}

\begin{theorem}
\label{IM}Let $M$ be a finitely generated multiplication $R$-module and $N=IM$
be a submodule of $M$.
\end{theorem}

\begin{enumerate}
\item If $M$ is torsion-free\textbf{ }and $N$ is a semi $n$-submodule of $M$,
then $I$ is a semi $n$-ideal of $R$.

\item If $R$ is an integral domain and $I$ is a semi $n$-ideal of $R$, then
$N$ is a semi $n$-submodule of $M$.
\end{enumerate}

\begin{proof}
(1) Suppose $N=IM$ is a semi $n$-submodule of $M$. Then $(N:_{R}%
M)=(IM:_{R}M)=I$ by Lemma \ref{Smith} and so, $I$ is a semi $n$-ideal by
Corollary \ref{(N:M)2}..

(2) Suppose that $R$ is an integral domain and $I$ is a semi $n$-ideal of $R$.
Note that $N=IM$ is proper in $M$ since otherwise by Lemma \ref{Smith}, we get
$I=(IM:_{R}M)=R$ which is a contradiction. Let $r\in R$ and $K=JM$ be a
nonzero submodule of $M$ such that $r^{2}K=r^{2}JM\subseteq IM$, $r\notin
\sqrt{0}$ and $T(K)=\{0_{M}\}.$ Take $A=rJ$ and note that $A^{2}%
\subseteq(r^{2}JM:M)\subseteq(IM:_{R}M)=I$ by Lemma \ref{Smith}. Let $a\in A$.
Then $a^{2}\in I$ and $a\notin\sqrt{0}$. Indeed, if $a=rj\in\sqrt{0}$, then
$0=r^{k}j^{k}M\subseteq r^{k}JM=r^{k}K$ for some $k\in%
\mathbb{N}
$. Since $K\neq0$ and $T(K)=\{0_{M}\}$, then $r^{k}=0$ which is a
contradiction. By assumption, we have $a\in I$ and so $A\subseteq I$.
Therefore, $rK=rJM=AM\subseteq IM$ and $N$ is a semi $n$-submodule of $M$ by
Theorem \ref{char}.
\end{proof}

In view of Corollary \ref{(N:M)2} and Theorem \ref{IM}, we conclude the
following relationship between semi $n$-submodules of a module $M$ and and
their residuals in $M$.

\begin{corollary}
\label{(N:M)}Let $R$ be a ring and $M$ be a finitely generated torsion-free
multiplication $R$-module. For a submodule $N$ of $M$, the following
statements are equivalent.

\begin{enumerate}
\item $N$ is a semi $n$-submodule of $M$.

\item $(N:_{R}M)$ is a semi $n$-ideal of $R$.

\item $N=IM$ for some semi $n$-ideal $I$ of $R$.
\end{enumerate}
\end{corollary}

We recall that for a submodule $N$ of an $R$-module $M$, $rad(N)$ denotes the
intersection of all prime submodules of $M$ containing $N$. Moreover, if $M$
is finitely generated faithful multiplication, then $rad(N)=\sqrt{(N:_{R}M)}%
M$, \cite{Smith}. One can conclude by Theorem \ref{IM} that if $R$ is an
integral domain, $M$ is a finitely generated multiplication $R$-module and $N$
is a submodule of $M$ such that $\sqrt{(N:_{R}M)}$ is a semi $n$-ideal of $R$,
then $rad(N)$ is a semi $n$-submodule of $M.$

Let $R$ be an integral domain and $I$ be an ideal of $R$. In the following
lemma, we show that if $N$ is a semi $n$-submodule of an $R$-module $M$ and
$(N:_{M}I)\neq M$, then $(N:_{M}I)$ is also a semi $n$-submodule of $M.$

\begin{lemma}
\label{(N:I)}Let $R$ be an integral domain and $N$ be a semi $n$-submodule of
an $R$-module $M$. Then for any ideal $I$ of $R$ with $(N:_{M}I)\neq M$,
$(N:_{M}I)$ is a semi $n$-submodule of $M.$ In particular, if $a\in R$ with
$(N:_{M}a)\neq M$, then $(N:_{M}a)$ is a semi $n$-submodule of $M.$
\end{lemma}

\begin{proof}
Suppose $N$ is a semi $n$-submodule of $M$. Let $r\in R$ and $K$ be a
submodule of $M$ such that $r^{2}K\subseteq(N:_{M}I)$, $r\notin\sqrt{0}$ and
$T(K)=\{0_{M}\}$. Then $r^{2}IK\subseteq N$ and clearly $T(IK)=\{0_{M}\}$. By
Theorem \ref{char}, we conclude that $rIK\subseteq N$ and so $rK\subseteq
(N:_{M}I)$.\ Therefore, $(N:_{M}I)$ is a semi $n$-submodule of $M$ again by
Theorem \ref{char}. The "in particular" part can be verified by a similar way.
\end{proof}

A submodule $N$ of an $R$-module $M$ is called a maximal semi $n$-submodule if
there is no proper submodule in $M$ which contains $N$ properly.

\begin{proposition}
Let $M$ be an $R$-module where $R$ is an integral domain. Then any maximal
semi $n$-submodule of $M$ is a prime submodule.
\end{proposition}

\begin{proof}
Suppose $N$ is a maximal semi $n$-submodule of an $R$-module $M$. Let $a\in
R,$ $m\in M$ with $am\in N$ and $a\notin(N:_{R}M).$ Then $(N:_{M}a)$ is
clearly proper in $M$ and so a semi $n$-submodule of $M$ by Lemma \ref{(N:I)}.
Since $N$ is maximal, we have $m\in(N:_{M}a)=N$. Thus, $N$ is a prime
submodule of $M.$
\end{proof}

Next, we discuss when $IN$ is a semi $n$-submodule of a finitely generated
multiplication module $M$ where $I$ is an ideal of $R$ and $N$ is a submodule
of $M$. Recall that a submodule $N$ of an $R$-module $M$ is said to be pure if
$JN=JM\cap N$ for every ideal $J$ of $R$. In the following definition, we give
a generalization of\ this concept.

\begin{definition}
Let $N$ be a submodule of an $R$-module $M$. Then $N$ is said to be weakly
pure if $JN=JM\cap rad(N)$ for every ideal $J$ of $R$.
\end{definition}

\begin{theorem}
\label{IN}Let $I$ be an ideal of an integral domain $R$, $M$ be a finitely
generated faithful multiplication $R$-module and $N$ be a proper submodule of
$M$.
\end{theorem}

\begin{enumerate}
\item If $I$ is a semi $n$-ideal of $R$ and $N$ is a weakly pure semi
$n$-submodule of $M$, then $IN$ is a semi $n$-submodule of $M$.

\item If $I$ is a finitely generated faithful multiplication ideal and $IN$ is
a semi $n$-submodule of $M$, then $N$ is a semi $n$-submodule of $M$.
\end{enumerate}

\begin{proof}
(1) We note that $IN$ is proper in $M$ since otherwise by Lemma \ref{Smith},
$R=(IN:_{R}M)=I(N:_{R}M)\subseteq I$, a contradiction. Suppose that
$r^{2}K\subseteq IN$, $r\notin\sqrt{0}$ and $T(K)=\{0_{M}\}$ for $r\in R$ and
a nonzero submodule $K=JM$ of $M$. Take $A=rJ$ and again use Lemma \ref{Smith}
to see that
\[
A^{2}\subseteq(r^{2}JM:_{R}M)\subseteq(IN:_{R}M)=I(N:_{R}M)\subseteq
I\cap(N:_{R}M)
\]
Let $a=rj\in A$ for $j\in J$ so that $a^{2}\in A^{2}\subseteq I$. If
$a\in\sqrt{0}$, then $0=r^{k}j^{k}M\subseteq r^{k}JM=r^{k}K$ for some $k\in%
\mathbb{N}
$. Since $K\neq0$ and $T(K)=\{0_{M}\}$, then $r^{k}=0$, a contradiction. Thus,
$a\notin\sqrt{0}$ and so $a\in I$ since $I$ is a semi $n$-ideal of $R$. Also,
we have $A\subseteq\sqrt{(N:_{R}M)}$ and so $A\subseteq I\cap\sqrt{(N:_{R}M)}%
$. Since $rad(N)=\sqrt{(N:_{R}M)}M$ and $N$ is weakly pure, we get
$rK=AM\subseteq IM\cap\sqrt{(N:_{R}M)}M=IM\cap rad(N)=IN$, as needed.

(2) Suppose that $IN$ is a semi $n$-submodule of $M$ where $I$ is finitely
generated faithful multiplication. If $N=M$, then by Lemma \ref{Majed},
$N=(IN:_{M}I)=(IM:_{M}I)=M$, a contradiction. Let $r\in R$ and $K$ be a
submodule of $M$ such that $r^{2}K\subseteq N$, $r\notin\sqrt{0}$ and
$T(K)=\{0_{M}\}$. Then $r^{2}IK\subseteq IN$ where clearly $T(IK)=\{0_{M}\}$.
By assumption, $rIK\subseteq IN$ and hence by Lemma \ref{Majed},
$rK\subseteq(IN:_{M}I)=N,$ as required.
\end{proof}

Next, we discuss the behavior of semi $n$-submodules under homomorphisms and localizations.

\begin{proposition}
\label{fsub}Let $M$ and $M^{\prime}$ be $R$-modules and $f:M\rightarrow
M^{\prime}$ be an $R$-module homomorphism.
\end{proposition}

\begin{enumerate}
\item If $f$ is an epimorphism and $N$ is a semi $n$-submodule of $M$
containing $Ker(f)$, then $f(N)$ is a semi $n$-submodule of $M^{\prime}$.

\item If $f$ is an isomorphism and $N^{\prime}$ is a semi $n$-submodule of
$M^{\prime}$, then $f^{-1}(N^{\prime})$ is a semi $n$-submodule of $M$.
\end{enumerate}

\begin{proof}
(1) Let $N$ be a semi $n$-submodule of $M$ and $r\in R$, $m^{\prime}\in
M^{\prime}$ such that $r^{2}m^{\prime}\in f(N)$, $r\notin\sqrt{0}$ and
$Ann_{R}(m^{\prime})=0$. Put $m^{\prime}=f(m)$ for some $m\in M.$ Then
$r^{2}f(m)\in f(N)$ which yields that $r^{2}m\in N$ as $Ker(f)\subseteq N$. If
$r\in Ann_{R}(m)$, then $rm=0_{M}$ which implies $rf(m)=0_{M^{\prime}}$. It
follows that $r\in Ann_{R}(m^{\prime})=0.$ Thus, $Ann_{R}(m)=0$ and so $rm\in
N$ as $N$ is a semi $n$-submodule of $M$. Therefore, $rm^{\prime}\in f(N)$ and
$f(N)$ is a semi $n$-submodule of $M^{\prime}$.

(2) Let $N^{\prime}$ be a semi $n$-submodule of $M^{\prime}$. Suppose that
$r^{2}m\in f^{-1}(N^{\prime})$, $r\notin\sqrt{0}$ and $Ann_{R}(m)=0$ for some
$r\in R$ and $m\in M$. Then $r^{2}f(m)=f(r^{2}m)\in N^{\prime}$. Assume that
$af(m)=0$ for some $a\in R.$ Then $f(am)=0$ implies $am\in Ker(f)=\{0_{M}\}$
and so $a\in Ann_{R}(m)=0$. Thus, $Ann_{R}(f(m))=0$ and since $N^{\prime}$ is
a semi $n$-submodule, we conclude that $rf(m)\in N^{\prime}$. Therefore,
$rm\in f^{-1}(N^{\prime})$ and we are done.
\end{proof}

Consequently, let $L\subseteq N$ be two submodules of an $R$-module $M$. If
$N$ is a semi $n$-submodule of $M$, then $N/L$ is a semi $n$-submodule of
$M/L.$ Indeed, consider the canonical epimorphism $\pi:M\rightarrow M/L$. Then
$Ker$ $\pi=L\subseteq N$ and $\pi(N)=N/L$is a semi $n$-submodule of $N/L$ by
(1) of Proposition \ref{fsub}.

Now, we investigate the relationships between semi $n$-submodules of an
$R$-module $M$ and those of the modules of fractions $S^{-1}M$ where $S$ is a
multiplicatively closed subset of $R$.

\begin{theorem}
\label{SM}Let $S$ be a multiplicatively closed subset of a ring $R$ and $M$ be
an $R$-module such that $S\subseteq reg(R)$.
\end{theorem}

\begin{enumerate}
\item If $N$ is a semi $n$-submodule of $M$ providing $%
{\displaystyle\bigcup\limits_{s\in S}}
$ $(N:_{M}s)\neq M$, then $S^{-1}N$ is a semi $n$-submodule of $S^{-1}M.$

\item If $S^{-1}N$ is a semi $n$-submodule of $S^{-1}R$ and $S\cap
Z_{N}(R)=\emptyset$, then $N$ is a semi $n$-submodule of $M.$
\end{enumerate}

\begin{proof}
(1) We note that $S^{-1}N$ is proper in $S^{-1}M$. Indeed, suppose
$S^{-1}N=S^{-1}M$ and let $m\in M$. Then $\frac{m}{1}\in S^{-1}N$ and so
$sm\in N$ for some $s\in S$. Hence, $m\in%
{\displaystyle\bigcup\limits_{s\in S}}
$ $(N:_{M}s)$, a contradiction. For $\frac{r}{s}\in S^{-1}R$ and $\frac{m}%
{t}\in S^{-1}M$,\ let $\left(  \frac{r}{s}\right)  ^{2}\left(  \frac{m}%
{t}\right)  \in S^{-1}N$ where $\frac{r}{s}\notin\sqrt{0_{S^{-1}R}}$ and
$Ann_{S^{-1}R}(\frac{m}{t})=0_{S^{-1}R}$. Choose $u\in S$ such that
$r^{2}(um)\in N$. Clearly, we have $r\notin\sqrt{0}$ and we show that
$Ann_{R}(um)=0$. Assume that $r^{\prime}um=0$ for some $r^{\prime}\in R$. Then
$\frac{r^{\prime}u}{1}\frac{m}{t}=0_{S^{-1}M}$ and since $Ann_{S^{-1}R}%
(\frac{m}{t})=0_{S^{-1}R}$, we conclude that $\frac{r^{\prime}u}{1}%
=0_{S^{-1}R}$. Thus, $r^{\prime}us=0$ for some $s\in S$. It follows that
$r^{\prime}=0$ since $us\in S\subseteq reg(R)$ and so $Ann_{R}(um)=0$. Since
$N$ is a semi $n$-submodule of $M$, $r^{2}(um)\in N,$ $r\notin\sqrt{0}$ and
$Ann_{R}(um)=0,$ we have $rum\in N$ and so $\frac{r}{s}\frac{m}{t}=\frac
{rum}{sut}\in S^{-1}N$. Thus, $S^{-1}N$ is a semi $n$-submodule of $S^{-1}M.$

(2) Suppose that $S^{-1}N$ is a semi $n$-submodule of $S^{-1}R.$ Clearly, $N$
is proper in $M$. Let $r\in R$ and $m\in M$ such that $r^{2}m\in N$,
$r\notin\sqrt{0}$ and $Ann_{R}(m)=0$. Then $\left(  \frac{r}{1}\right)
^{2}\frac{m}{1}\in S^{-1}N$ and $\frac{r}{1}\notin\sqrt{0_{S^{-1}R}}$. Indeed,
if there exists an integer $k$ such that $\left(  \frac{r}{1}\right)
^{k}=\frac{0}{1}$, then $ur^{k}=0$ for some $u\in S$. Thus, $r^{k}=0$ as
$S\subseteq reg(R)$ which is a contradiction. Now, let $\frac{r}{s}\in
Ann_{S^{-1}R}(\frac{m}{1})$ so that $\frac{r}{s}\frac{m}{1}=0_{S^{-1}M}$.
Thus, $rvm=0$ for some $v\in S$ and so $rv=0$ as $Ann_{R}(m)=0$. Since
$S\subseteq reg(R)$, we get $r=0$ and so $\frac{r}{s}=\frac{0}{1}$. Hence,
$Ann_{S^{-1}R}(\frac{m}{1})=0_{S^{-1}R}$ and by assumption, we conclude that
$\frac{r}{1}\frac{m}{1}\in S^{-1}N$. Hence, $wrm\in N$ for some $w\in S$ and
since $S\cap Z_{N}(M)=\emptyset$, we conclude that $rm\in N$, as desired.
\end{proof}

The proof of the following Lemma is straightforward.

\begin{lemma}
\label{int}Let $\{N_{i}\}_{i\in I}$ be a non-empty family of semi
$n$-submodules of an $R$-module $M$. Then $%
{\displaystyle\bigcap\limits_{i\in I}}
N_{i}$ is a semi $n$-submodule of $M$. Additionally, $%
{\displaystyle\bigcup\limits_{i\in I}}
N_{i}$ is a semi $n$-submodule of $M$ provided that $\{N_{i}\}_{i\in I}$ is a
chain in $M$.
\end{lemma}

Now, for a ring $R$, we examine the semi $n$-submodules of the finite
Cartesian product of $R$-modules.

\begin{theorem}
\label{cart}Let $M_{1},M_{2},$ $\ldots,M_{k}$ be $R$-modules and consider the
$R$-module $M=M_{1}\times M_{2}\times\cdots\times M_{k}$. Let $N_{1}%
,N_{2},\ldots,N_{k}$ be submodules of $M_{1},M_{2},$ $\ldots,M_{k}$,
respectively. If $N=N_{1}\times N_{2}\times\cdots\times N_{k}$ is a semi
$n$-submodule of $M$, then $N_{i}$ is a semi $n$-submodule of $M_{i}$ whenever
$N_{i}\neq M_{i}$ ($i=1,2,\ldots,k$). The converse also holds if $M_{i}$\ is
torsion-free whenever $N_{i}\neq M_{i}$ ($i=1,2,\ldots,k$).
\end{theorem}

\begin{proof}
Suppose $N$ is a semi $n$-submodule of $M$ and $N_{i}\neq M_{i}$ for some
$i=1,2,\ldots,k$. Let $r\in R$, $m_{i}\in M_{i}$ with $r^{2}m_{i}\in N_{i}$,
$r\notin\sqrt{0}$ and $Ann_{R}(m_{i})=0.$ Then $r^{2}(0,\ldots,m_{i}%
,\ldots,0)\in N$ and $Ann_{R}((0,\ldots,m_{i},\ldots,0))=0$. Since $N$ is a
semi $n$-submodule of $M$, then $r(0,\ldots,m_{i},\ldots,0)\in N$ and so
$rm_{i}\in N_{i}$. Thus, $N_{i}$ is a semi $n$-submodule of $M_{i}.$

Conversely, suppose $M_{i}$\ is torsion-free whenever $N_{i}\neq M_{i}$
($i=1,2,\ldots,k$). Let $r^{2}(m_{1},m_{2},\ldots,m_{k})\in N$, $r\notin
\sqrt{0}$ and $Ann_{R}((m_{1},m_{2},\ldots,m_{k}))=0$. If $N_{i}\neq M_{i}$
($i=1,2,\ldots,k)$, then $r^{2}m_{i}\in N_{i}$, $r\notin\sqrt{0}$ and
$T(M_{i})=0$. By assumption, $rm_{i}\in N_{i}$ and so $r(m_{1},m_{2}%
,\ldots,m_{k})\in N$.
\end{proof}

\begin{corollary}
\label{cc}Let $M_{1}$ and $M_{2}$ be $R$-modules and consider the $R$-module
$M_{1}\times M_{2}$. Let $N_{1}$ and $N_{2}$ be proper submodules of $M_{1}$
and $M_{2}$, respectively. If $N_{1}\times N_{2}$ is a semi $n$-submodule of
$M_{1}\times M_{2}$, then $N_{1}$ is a semi $n$-submodule of $M_{1}$ and
$N_{2}$ is a semi $n$-submodule of $M_{2}$. The converse also holds if $M_{1}$
and $M_{2}$ are torsion-free.
\end{corollary}

Let $M$ be an $R$-module. We recall from \cite{Ande} that the idealization of
$M$ by~$R$ is the commutative ring $R\times M$ with coordinate-wise addition
and multiplication defined as $(r_{1},m_{1})(r_{2},m_{2})=(r_{1}r_{2}%
,r_{1}m_{2}+r_{2}m_{1})$, denoted by $R(+)M$. For an ideal $I$ of $R$ and a
submodule $N$ of $M$, $I(+)N$ is an ideal of $R(+)M$ if and only if
$IM\subseteq N$. Also, $\sqrt{0_{R(+)M}}=\sqrt{0}(+)M$. It is proved in
\cite{HE3} that for a proper ideal $I$ of a ring $R$, we have $I$ is a semi
$n$-ideal of $R$ if and only if $I(+)M$ is a semi $n$-ideal of $R(+)M$. For an
ideal $I$ of a ring $R$ and a submodule $N$ of $M$, we justify in the
following when is the ideal $I(+)N$ a semi $n$-ideal of $R(+)M$.

\begin{theorem}
\label{Ide}Let $I$ be a proper ideal of a ring $R$ and $N$ be a submodule of
an $R$-module $M$ such that $IM\subseteq N$. If $I(+)N$ is a semi $n$-ideal of
$R(+)M$, then $I$ is a semi $n$-ideal of $R$ and $N$ is an $n$-submodule of
$M$. Moreover, the converse is true if $\sqrt{Ann_{R}(M)}=\sqrt{0}$.
\end{theorem}

\begin{proof}
Assume that $I(+)N$ is a semi $n$-ideal of $R(+)M$. Let $r\in R$ such that
$r^{2}\in I$ but $r\notin\sqrt{0}$. Then $(r,0_{M})^{2}\in I(+)N$ and
$(r,0_{M})\notin\sqrt{0}(+)M=\sqrt{0_{R(+)M}}$. Thus, $(r,0_{M})\in I(+)N$ and
so $r\in I$, as needed. Now, let $r\in R$ and $m\in N$ such that $rm\in N$ and
$r\notin\sqrt{Ann_{R}(M)}$. Then $(r,0_{M})(0,m)\in I(+)N$ with clearly
$(r,0_{M})\notin\sqrt{0_{R(+)M}}$. It follows that $(0,m)\in I(+)N$ and so
$m\in N$. Therefore, $I$ is a semi $n$-ideal of $R$ and $N$ is an
$n$-submodule of $M$. Conversely, suppose $\sqrt{Ann_{R}(M)}=\sqrt{0}$. Let
$(r,m)\in R(+)M$ such that $(r,m)^{2}\in I(+)N$ and $(r,m)\notin
\sqrt{0_{R(+)M}}=\sqrt{0}(+)M$. Then $r^{2}\in I$ with $r\notin\sqrt{0}$
implies $r\in I$. Also, we have $rm\in N$ as $IM\subseteq N$ and since
$\sqrt{Ann_{R}(M)}=\sqrt{0}$, $r\notin\sqrt{Ann_{R}(M)}$. By assumption, $m\in
N$ and so $(r,m)\in I(+)N$.
\end{proof}

\begin{remark}
In general, if $\sqrt{Ann_{R}(M)}\neq\sqrt{0}$, then the converse of
Proposition \ref{Ide} need not be true. For example, consider the idealization
ring $R=%
\mathbb{Z}
(+)%
\mathbb{Z}
_{4}$ and the ideal $2%
\mathbb{Z}
(+)\left\langle \bar{2}\right\rangle $ of $R$. Then $2%
\mathbb{Z}
$ is a semi $n$-ideal of $%
\mathbb{Z}
$ by \cite[Example 2.1]{HE3} and $\left\langle \bar{2}\right\rangle $ is an
$n$-submodule of $%
\mathbb{Z}
_{4}$. Indeed, if $rm\in\left\langle \bar{2}\right\rangle $ where
$r\notin\sqrt{Ann_{%
\mathbb{Z}
}(%
\mathbb{Z}
_{4})}=2%
\mathbb{Z}
$, then clearly $m\in\left\langle \bar{2}\right\rangle $ as needed. On the
other hand, $2%
\mathbb{Z}
(+)\left\langle \bar{2}\right\rangle $ is not a semi $n$-ideal of $R$ since
for example $(2,\bar{1})^{2}=(4,\bar{0})\in2%
\mathbb{Z}
(+)\left\langle \bar{2}\right\rangle $ but $(2,\bar{1})\notin\sqrt{0_{R}}=0(+)%
\mathbb{Z}
_{4}$ and $(2,\bar{1})\notin2%
\mathbb{Z}
(+)\left\langle \bar{2}\right\rangle $. Note that $2%
\mathbb{Z}
=\sqrt{Ann_{%
\mathbb{Z}
}(%
\mathbb{Z}
_{4})}\neq\sqrt{0}=0$.
\end{remark}

\section{Semi $n$-submodules of amalgamated modules}

Let $R$ be a ring, $J$ an ideal of $R$ and $M$ an $R$-module. Recently, in
\cite{Bouba}, the duplication of the $R$-module $M$ along the ideal $J$
(denoted by $M\Join J$) is defined as%
\[
M\Join J=\left\{  (m,m^{\prime})\in M\times M:m-m^{\prime}\in JM\right\}
\]
which is an $(R\Join J)$-module with scalar multiplication defined by
$(r,r+j)\cdot(m,m^{\prime})=(rm,(r+j)m^{\prime})$ for $r\in R$, $j\in J$ and
$(m,m^{\prime})\in M\Join J$. For various properties and results concerning
this kind of modules, one may refer to \cite{Bouba}.

Let $J$ be an ideal of a ring $R$ and $N$ be a submodule of an $R$-module $M$. Then%

\[
N\Join J=\left\{  (n,m)\in N\times M:n-m\in JM\right\}
\]
and
\[
\bar{N}=\left\{  (m,n)\in M\times N:m-n\in JM\right\}
\]
are clearly submodules of $M\Join J$. Moreover,%

\[
Ann_{R\Join J}(M\Join J)=(r,r+j)\in R\Join I\text{ }|\text{ }r\in
Ann_{R}(M)\text{ and }j\in Ann_{R}(M)\cap J\}
\]
and so $M\Join J$ is a faithful $R\Join J$ -module if and only if $M$ is a
faithful $R$-module, \cite[Lemma 3.6]{Bouba}.

In general, let $f:R_{1}\rightarrow R_{2}$ be a ring homomorphism, $J$ be an
ideal of $R_{2}$, $M_{1}$ be an $R_{1}$-module, $M_{2}$ be an $R_{2}$-module
(which is an $R_{1}$-module induced naturally by $f$) and $\varphi
:M_{1}\rightarrow M_{2}$ be an $R_{1}$-module homomorphism. The subring
\[
R_{1}\Join^{f}J=\left\{  (r,f(r)+j):r\in R_{1}\text{, }j\in J\right\}
\]
of $R_{1}\times R_{2}$ is called the amalgamation of $R_{1}$ and $R_{2}$ along
$J$ with respect to $f$. In \cite{Rachida}, the amalgamation of $M_{1}$ and
$M_{2}$ along $J$ with respect to $\varphi$ is defined as%

\[
M_{1}\Join^{\varphi}JM_{2}=\left\{  (m_{1},\varphi(m_{1})+m_{2}):m_{1}\in
M_{1}\text{ and }m_{2}\in JM_{2}\right\}
\]
which is an $(R_{1}\Join^{f}J)$-module with the scalar product defined as
\[
(r,f(r)+j)(m_{1},\varphi(m_{1})+m_{2})=(rm_{1},\varphi(rm_{1})+f(r)m_{2}%
+j\varphi(m_{1})+jm_{2})
\]
For submodules $N_{1}$ and $N_{2}$ of $M_{1}$ and $M_{2}$, respectively, one
can easily justify that the sets
\[
N_{1}\Join^{\varphi}JM_{2}=\left\{  (m_{1},\varphi(m_{1})+m_{2})\in M_{1}%
\Join^{\varphi}JM_{2}:m_{1}\in N_{1}\right\}
\]
and
\[
\overline{N_{2}}^{\varphi}=\left\{  (m_{1},\varphi(m_{1})+m_{2})\in M_{1}%
\Join^{\varphi}JM_{2}:\text{ }\varphi(m_{1})+m_{2}\in N_{2}\right\}
\]
are submodules of $M_{1}\Join^{\varphi}JM_{2}$.

Note that if $R=R_{1}=R_{2}$, $M=M_{1}=M_{2}$, $f=Id_{R}$ and $\varphi=Id_{M}%
$, then the amalgamation of $M_{1}$ and $M_{2}$ along $J$ with respect to
$\varphi$ is exactly the duplication of the $R$-module $M$ along the ideal
$J$. Moreover, in this case, we have $N_{1}\Join^{\varphi}JM_{2}=N\Join J$ and
$\overline{N_{2}}^{\varphi}=\bar{N}$.

The proof of the following lemma is straightforward.

\begin{lemma}
\label{Lemma}Consider the ring $R_{1}\Join^{f}J$ as above. Then $\sqrt
{0_{R\Join^{f}J}}=\sqrt{0_{R_{1}}}\Join^{f}J$ if and only if $J\subseteq
\sqrt{0_{R_{2}}}$.
\end{lemma}

In the following theorems, we justify conditions under which $N_{1}%
\Join^{\varphi}JM_{2}$ and $\overline{N_{2}}^{\varphi}$ are $n$-submodules
(semi $n$-submodule) in $M_{1}\Join^{\varphi}JM_{2}$. Note that clearly
$N_{1}$ is proper in $M_{1}$ if and only if $N_{1}\Join^{\varphi}JM_{2}$ is
proper in $M_{1}\Join^{\varphi}JM_{2}$.

\begin{theorem}
\label{Amalg}Consider the $(R_{1}\Join^{f}J)$-module $M_{1}\Join^{\varphi
}JM_{2}$ defined as above and let $N_{1}$ be a proper submodule of $M_{1}$. If
$N_{1}\Join^{\varphi}JM_{2}$ is an $n$-submodule of $M_{1}\Join^{\varphi
}JM_{2}$, then $N_{1}$ is an $n$-submodule of $M_{1}$. Moreover, the converse
is true if $JM_{2}=\left\{  0_{M_{2}}\right\}  $.
\end{theorem}

\begin{proof}
Let $r_{1}\in R_{1}$ and $m_{1}\in M_{1}$ such that $r_{1}m_{1}\in N_{1}$ and
$r_{1}\notin\sqrt{Ann_{R_{1}}(M_{1})}$. Then $(r_{1},f(r_{1}))\in R_{1}%
\Join^{f}J$ , $(m_{1},\varphi(m_{1}))\in M_{1}\Join^{\varphi}JM_{2}$ and
$(r_{1},f(r_{1}))(m_{1},\varphi(m_{1}))=(r_{1}m_{1},\varphi(r_{1}m_{1}))\in
N_{1}\Join^{\varphi}JM_{2}$. Moreover, $(r_{1},f(r_{1}))\notin\sqrt
{Ann_{R_{1}\Join^{f}J}(M_{1}\Join^{\varphi}JM_{2})}$. Indeed, suppose that
there is a positive integer $k$ such that $(r_{1},f(r_{1}))^{k}(M_{1}%
\Join^{\varphi}JM_{2})=(0_{M_{1}},0_{M_{2}})$. Then $r_{1}^{k}M_{1}=0$ and so
$r_{1}\in\sqrt{Ann_{R_{1}}(M_{1})}$, a contradiction. Since $N_{1}%
\Join^{\varphi}JM_{2}$ is an $n$-submodule of $M_{1}\Join^{\varphi}JM_{2}$,
then $(m_{1},\varphi(m_{1}))\in N_{1}\Join^{\varphi}JM_{2}$ and so $m_{1}\in
N_{1},$ as needed. Conversely suppose $JM_{2}=\left\{  0_{M_{2}}\right\}  $
and $N_{1}$ is an $n$-submodule of $M_{1}$. Let $(r_{1},f(r_{1})+j)\in
R_{1}\Join^{f}J$, $(m_{1},\varphi(m_{1}))\in M_{1}\Join^{\varphi}JM_{2}$ such
that $(r_{1},f(r_{1})+j)(m_{1},\varphi(m_{1}))\in N_{1}\Join^{\varphi}JM_{2}$
and $(r_{1},f(r_{1})+j)\notin\sqrt{Ann_{R_{1}\Join^{f}J}(M_{1}\Join^{\varphi
}JM_{2})}$. Then $r_{1}m_{1}\in N_{1}$ and we prove that $r_{1}\notin
\sqrt{Ann_{R_{1}}(M_{1})}$. Suppose $r_{1}^{k}M_{1}=0_{M_{1}}$ for some
positive integer $k$. Then for any $(m_{1},\varphi(m_{1}))\in M_{1}%
\Join^{\varphi}JM_{2}$, we have
\begin{align*}
(r_{1},f(r_{1})+j)^{k}(m_{1},\varphi(m_{1}))  &  =(r_{1}^{k},f(r_{1}%
^{k})+j^{\prime})(m_{1},\varphi(m_{1}))\\
&  =(0_{M_{1}},j^{\prime}\varphi(m_{1}))=(0_{M_{1}},0_{M_{2}})
\end{align*}
for some $j^{\prime}\in J$ as $JM_{2}=\left\{  0_{M_{2}}\right\}  $. Thus,
$(r_{1},f(r_{1})+j)\notin\sqrt{Ann_{R_{1}\Join^{f}J}(M_{1}\Join^{\varphi
}JM_{2})}$, a contradiction. By assumption, we conclude that $m_{1}\in N_{1}$
and so $(m_{1},\varphi(m_{1}))\in N_{1}\Join^{\varphi}JM_{2}$, as needed.
\end{proof}

\begin{theorem}
Consider the $(R_{1}\Join^{f}J)$-module $M_{1}\Join^{\varphi}JM_{2}$ defined
as above where $JM_{2}=\left\{  0_{M_{2}}\right\}  $.
\end{theorem}

\begin{enumerate}
\item If $J\subseteq\sqrt{0_{R_{2}}}$ and $N_{1}$ is a semi $n$-submodule of
$M_{1}$, then $N_{1}\Join^{\varphi}JM_{2}$ is a semi $n$-submodule of
$M_{1}\Join^{\varphi}JM_{2}$.

\item If $M_{2}$ is faithful and $N_{1}\Join^{\varphi}JM_{2}$ is a semi
$n$-submodule of $M_{1}\Join^{\varphi}JM_{2}$, then $N_{1}$ is a semi
$n$-submodule of $M_{1}$.
\end{enumerate}

\begin{proof}
(1) Suppose $J\subseteq\sqrt{0_{R_{2}}}$ and $N_{1}$ is a semi $n$-submodule
of $M_{1}$. Let $(r_{1},f(r_{1})+j)\in R_{1}\Join^{f}J$ and $(m_{1}%
,\varphi(m_{1}))\in M_{1}\Join^{\varphi}JM_{2}$ such that $(r_{1}%
,f(r_{1})+j)^{2}(m_{1},\varphi(m_{1}))\in N_{1}\Join^{\varphi}JM_{2}$,
$(r_{1},f(r_{1})+j)\notin\sqrt{0_{R_{1}\Join^{f}J}}$ and $Ann_{R_{1}\Join
^{f}J}((m_{1},\varphi(m_{1})))=0_{R_{1}\Join^{f}J}$. Then $r_{1}^{2}m_{1}\in
N_{1}$ and $r_{1}\notin\sqrt{0_{R_{1}}}$ since $\sqrt{0_{R_{1}\Join^{f}J}%
}=\sqrt{0_{R_{1}}}\Join^{f}J$ by Lemma \ref{Lemma}. We show that $Ann_{R_{1}%
}(m_{1})=0_{R_{1}}$. Let $r_{1}^{\prime}\in R_{1}$ such that $r_{1}^{\prime
}m_{1}=0_{M_{1}}.$ Then, $(r_{1}^{\prime},f(r_{1}^{\prime}))(m_{1}%
,\varphi(m_{1}))=0_{M_{1}\Join^{\varphi}JM_{2}}$ and since $Ann_{R_{1}%
\Join^{f}J}((m_{1},\varphi(m_{1})))=0_{R_{1}\Join^{f}J}$, we get
$(r_{1}^{\prime},f(r_{1}^{\prime}))=0_{R_{1}\Join^{f}J}$. Thus, $r_{1}%
^{\prime}=0_{R_{1}}$and so $Ann_{R_{1}}(m_{1})=0_{R_{1}}$. It follows that
$r_{1}m_{1}\in N_{1}$ and so $(r_{1},f(r_{1})+j)(m_{1},\varphi(m_{1}))\in
N_{1}\Join^{\varphi}JM_{2}$.

(2) Suppose $M_{2}$ is faithful and $N_{1}\Join^{\varphi}JM_{2}$ is a semi
$n$-submodule of $M_{1}\Join^{\varphi}JM_{2}$. Then clearly $J=\left\{
0_{R_{2}}\right\}  $. Let $r_{1}\in R_{1}$ and $m_{1}\in M_{1}$ such that
$r_{1}^{2}m_{1}\in N_{1}$, $r_{1}\notin\sqrt{0_{R_{1}}}$ and $Ann_{R_{1}%
}(m_{1})=0_{R_{1}}$. Then $(r_{1},f(r_{1}))^{2}(m_{1},\varphi(m_{1}))\in
N_{1}\Join^{\varphi}JM_{2}$ where $(r_{1},f(r_{1}))\in R_{1}\Join^{f}J$ and
$(m_{1},\varphi(m_{1}))\in M_{1}\Join^{\varphi}JM_{2}$. Moreover, clearly
$(r_{1},f(r_{1}))\notin\sqrt{0_{R_{1}\Join^{f}J}}$. Now, let $(r_{1}^{\prime
},f(r_{1}^{\prime}))\in R_{1}\Join^{f}J$ such that $(r_{1}^{\prime}%
m_{1},\varphi(r_{1}^{\prime}m_{1}))=(r_{1}^{\prime},f(r_{1}^{\prime}%
))(m_{1},\varphi(m_{1}))=0_{M_{1}\Join^{\varphi}JM_{2}}$. Then $(r_{1}%
^{\prime},f(r_{1}^{\prime}))=(0_{R_{1}},0_{R_{2}})$ as $Ann_{R_{1}}%
(m_{1})=0_{R_{1}}$ and so $Ann_{R_{1}\Join^{f}J}((m_{1},\varphi(m_{1}%
)))=0_{R_{1}\Join^{f}J}$. By assumption, $(r_{1},f(r_{1}))(m_{1},\varphi
(m_{1}))\in N_{1}\Join^{\varphi}JM_{2}$. It follows that $r_{1}m_{1}\in N_{1}$
and $N_{1}$ is a semi $n$-submodule of $M_{1}$.
\end{proof}

\begin{corollary}
\label{Dup1}Let $N$ be a submodule of an $R$-module $M$ and $J$ be an ideal of
$R$. Then
\end{corollary}

\begin{enumerate}
\item If $N\Join J$ is an $n$-submodule of $M\Join J$, then $N$ is an
$n$-submodule of $M$. The converse is true if $JM=0_{M}$.

\item If $N\Join J$ is a semi $n$-submodule of $M\Join J$, then $N$ is a semi
$n$-submodule of $M$. The converse is true if $J\subseteq\sqrt{0}\cap
Ann_{R}(M)$.
\end{enumerate}

\begin{proof}
(1) Suppose $N\Join J$ is an $n$-submodule of $M\Join J$. Let $r\in R$ and
$m\in M$ such that $rm\in N$ and $r\notin\sqrt{Ann_{R}(M)}$. Then $(r,r)\in
R\Join J$, $(m,m)\in M\Join J$, $(r,r)(m,m)\in$ $N\Join J$ and clearly,
$(r,r)\notin\sqrt{Ann_{R\Join J}(M\Join J)}$. Since $N\Join J$ is an
$n$-submodule of $M\Join J$, then $(m,m)\in N\Join J$ and so $m\in N$ as
needed. Conversely, suppose $JM=0_{M}$ and let $(r,r+j)\in R\Join J$,
$(m,m)\in M\Join J$ such that $(r,r+j)(m,m)\in N\Join J$ and $(r,r+j)\notin
\sqrt{Ann_{R\Join J}(M\Join J)}$. Then $rm\in N$ and $r\notin\sqrt{Ann_{R}%
(M)}$. Indeed, if $r^{k}M=0_{M}$ for some $k\in%
\mathbb{N}
$, then clearly, $(r,r+j)^{k}(M\Join J)=0_{M\Join J}$ as $JM=0_{M}$. Since $N$
is an $n$-submodule of $M$, then $m\in N$ and so $(m,m)\in N\Join J$.

(2) Suppose $N\Join J$ is a semi $n$-submodule of $M\Join J$. Let $r\in R$ and
$m\in M$ such that $r^{2}m\in N$, $r\notin\sqrt{0}$ and $Ann_{R}(m)=0_{R}$.
Then $(r,r)\in R\Join J$, $(m,m)\in M\Join J$ and $(r,r)^{2}(m,m)\in N\Join
J$. Moreover, clearly $(r,r)\notin\sqrt{0_{R\Join J}}$. Let $(r^{\prime
},r^{\prime}+j)\in Ann_{R\Join J}((m,m))$ so that $(r^{\prime},r^{\prime
}+j)(m,m)=(0_{M},0_{M})$. Then $(r^{\prime},r^{\prime}+j)=(0_{R},0_{R})$ since
$Ann_{R}(m)=0_{R}$. By assumption, $(r,r)(m,m)\in N\Join J$ and so $rm\in N$.
Conversely, suppose $J\subseteq\sqrt{0}\cap Ann_{R}(M)$ and $N$ is a semi
$n$-submodule of $M$. Let $(r,r+j)\in R\Join J$ and $(m,m)\in M\Join J$ such
that $(r,r+j)^{2}(m,m)\in N\Join J$, $(r,r+j)\notin\sqrt{0_{R\Join J}}$ and
$Ann_{R\Join J}(m,m)=0_{R\Join J}$. Then $r^{2}m\in N$ and $r\notin\sqrt{0}$
by Lemma \ref{Lemma}. Moreover, if $r^{\prime}m=0$ for some $r^{\prime}\in R$,
then $(r^{\prime},r^{\prime}+j)(m,m)=(0_{M},0_{M})$ as $JM=0_{M}$. Thus,
$(r^{\prime},r^{\prime}+j)=(0,0)$ and so $r^{\prime}=0$. Hence, $Ann_{R}(m)=0$
and by assumption, we conclude that $rm\in N$. Therefore, $(r,r+j)(m,m)\in
N\Join J$ and $N\Join J$ is a semi $n$-submodule of $M\Join J$.
\end{proof}

\begin{theorem}
\label{Amalg2}Consider the $(R_{1}\Join^{f}J)$-module $M_{1}\Join^{\varphi
}JM_{2}$ defined as in Theorem \ref{Amalg} and let $N_{2}$ be a submodule of
$M_{2}$.
\end{theorem}

\begin{enumerate}
\item If $N_{2}$ is an $n$-submodule of $M_{2}$, $JM_{2}=\left\{  0_{M_{2}%
}\right\}  $ and $\varphi$ is an isomorphism, then $\overline{N_{2}}^{\varphi
}$ is an $n$-submodule of $M_{1}\Join^{\varphi}JM_{2}$.

\item If $f$ and $\varphi$ are epimorphisms and $\overline{N_{2}}^{\varphi}$
is an $n$-submodule of $M_{1}\Join^{\varphi}JM_{2}$, then $N_{2}$ is an
$n$-submodule of $M_{2}$.

\item If $f$ is an isomorphism, $\varphi$ is an epimorphism and $\overline
{N_{2}}^{\varphi}$ is a semi $n$-submodule of $M_{1}\Join^{\varphi}JM_{2}$,
then $N_{2}$ is a semi $n$-submodule of $M_{2}$.
\end{enumerate}

\begin{proof}
(1) Suppose $N_{2}$ is an $n$-submodule of $M_{2}$. Suppose $\overline{N_{2}%
}^{\varphi}=M_{1}\Join JM_{2}$ and let $m_{2}=\varphi(m_{1})\in M_{2}$. Then
$(m_{1},m_{2})\in M_{1}\Join JM_{2}=\overline{N_{2}}^{\varphi}$ and so
$m_{2}\in N_{2}$. Thus, $N_{2}=M_{2}$, a contradiction. Therefore,
$\overline{N_{2}}^{\varphi}$ is proper in $M_{1}\Join JM_{2}$. Let
$(r_{1},f(r_{1})+j)\in R_{1}\Join^{f}J$ and $(m_{1},\varphi(m_{1})+m_{2})\in
M_{1}\Join JM_{2}$ such that $(r_{1},f(r_{1})+j)(m_{1},\varphi(m_{1}%
)+m_{2})\in\overline{N_{2}}^{\varphi}$ and $(r_{1},f(r_{1})+j)\notin
\sqrt{Ann_{R_{1}\Join^{f}J}(M_{1}\Join^{\varphi}JM_{2})}$. Then $(f(r_{1}%
)+j)(\varphi(m_{1})+m_{2})\in N_{2}$ and we prove that $f(r_{1})+j\notin
\sqrt{Ann_{R_{2}}(M_{2})}$. Suppose on the contrary that $(f(r_{1}%
)+j)^{k}M_{2}=0_{M_{2}}$ for some $k\in%
\mathbb{N}
$ and let $(m_{1}^{\prime},\varphi(m_{1}^{\prime})+m_{2}^{\prime})\in
M_{1}\Join^{\varphi}JM_{2}$. Then $(f(r_{1})+j)^{k}\varphi(m_{1}^{\prime
})=\varphi(r_{1}^{k}m_{1}^{\prime})+j^{\prime}\varphi(m_{1}^{\prime}%
)=0_{M_{2}}$\ for some $j^{\prime}\in J$ and so $r_{1}^{k}m_{1}^{\prime
}=0_{M_{1}}$ since $JM_{2}=0_{M_{2}}$ and $\varphi$ is one to one. Thus,
$(r_{1},f(r_{1})+j)^{k}(m_{1}^{\prime},\varphi(m_{1}^{\prime})+m_{2}^{\prime
})=0_{M_{1}\Join^{\varphi}JM_{2}}$ which is a contradiction. By assumption, we
have $\varphi(m_{1})+m_{2}\in N_{2}$ and so $(m_{1},\varphi(m_{1})+m_{2}%
)\in\overline{N_{2}}^{\varphi}$.

(2) Suppose $f$ and $\varphi$ are epimorphisms and $\overline{N_{2}}^{\varphi
}$ is an $n$-submodule of $M_{1}\Join^{\varphi}JM_{2}$. Clearly, $N_{2}$ is
proper in $M_{2}$. Let $r_{2}=f(r_{1})\in R_{2}$ and $m_{2}=\varphi(m_{1})\in
M_{2}$ such that $r_{2}m_{2}\in N_{2}$ and $r_{2}\notin\sqrt{Ann_{R_{2}}%
(M_{2})}$. Then $(r_{1},r_{2})\in R_{1}\Join^{f}J$, $(m_{1},m_{2})\in
M_{1}\Join^{\varphi}JM_{2}$ and $(r_{1},r_{2})(m_{1},m_{2})\in\overline{N_{2}%
}^{\varphi}$. Suppose on contrary that $(r_{1},r_{2})\in\sqrt{Ann_{R_{1}%
\Join^{f}J}(M_{1}\Join^{\varphi}JM_{2})}$ so that $(r_{1},r_{2})^{k}%
(M_{1}\Join^{\varphi}JM_{2})=0_{M_{1}\Join^{\varphi}JM_{2}}$ for some $k\in%
\mathbb{N}
$. Let $m_{2}^{\prime}=\varphi(m_{1}^{\prime})\in M_{2}$. Then $(r_{1}%
,r_{2})^{k}(m_{1}^{\prime},m_{2}^{\prime})=0_{M_{1}\Join^{\varphi}JM_{2}}$ and
so $r_{2}^{k}m_{2}^{\prime}=0_{M_{2}}$. Thus, $r_{2}\notin\sqrt{Ann_{R_{2}%
}(M_{2})}$ which is a contradiction. Therefore, $(r_{1},r_{2})\notin
\sqrt{Ann_{R_{1}\Join^{f}J}(M_{1}\Join^{\varphi}JM_{2})}$ and by assumption,
we have $(m_{1},m_{2})\in\overline{N_{2}}^{\varphi}$. It follows that
$m_{2}\in N_{2}$ as needed.

(3) Similar to the proof of (2).
\end{proof}

\begin{theorem}
Consider the $(R_{1}\Join^{f}J)$-module $M_{1}\Join^{\varphi}JM_{2}$ defined
as in Theorem \ref{Amalg} where $f$ is an isomorphism and $\varphi$ is an
epimorphism. Let $N_{2}$ be a submodule of $M_{2}$.
\end{theorem}

\begin{enumerate}
\item If $\overline{N_{2}}^{\varphi}$ is a semi $n$-submodule of $M_{1}%
\Join^{\varphi}JM_{2}$, then $N_{2}$ is a semi $n$-submodule of $M_{2}$.

\item If $J\subseteq\sqrt{0}\cap Ann_{R}(M)$ and $N_{2}$ is a semi
$n$-submodule of $M_{2}$, then $\overline{N_{2}}^{\varphi}$ is a semi
$n$-submodule of $M_{1}\Join^{\varphi}JM_{2}$.
\end{enumerate}

\begin{proof}
(1) Suppose $\overline{N_{2}}^{\varphi}$ is a semi $n$-submodule of
$M_{1}\Join^{\varphi}JM_{2}$. Let $r_{2}=f(r_{1})\in R_{2}$ and $m_{2}%
=\varphi(m_{1})\in M_{2}$ such that $r_{2}^{2}m_{2}\in N_{2}$, $r_{2}%
\notin\sqrt{0_{R_{2}}}$ and $Ann_{R_{2}}(m_{2})=0_{R_{2}}$. Then $(r_{1}%
,r_{2})^{2}(m_{1},m_{2})\in\overline{N_{2}}^{\varphi}$ where $(r_{1},r_{2})\in
R_{1}\Join^{f}J$, $(m_{1},m_{2})\in M_{1}\Join^{\varphi}JM_{2}$ and clearly
$(r_{1},r_{2})\notin\sqrt{0_{R_{1}\Join^{f}J}}$. We prove that $Ann_{R_{1}%
\Join^{f}J}((m_{1},m_{2}))=0_{R_{1}\Join^{f}J}$. Let $(r_{1}^{\prime}%
,f(r_{1}^{\prime})+j^{\prime})\in R_{1}\Join^{f}J$ such that $(r_{1}^{\prime
},f(r_{1}^{\prime})+j^{\prime})(m_{1},m_{2})=0_{M_{1}\Join^{\varphi}JM_{2}}$.
Then $r_{1}^{\prime}m_{1}=0_{M_{1}}$ and $(f(r_{1}^{\prime})+j^{\prime}%
)m_{2}=0_{M_{2}}=f(r_{1}^{\prime})m_{2}=0_{M_{2}}$ and so $(f(r_{1}^{\prime
})+j^{\prime})=f(r_{1}^{\prime})=0_{R_{2}}$ as $Ann_{R_{2}}(m_{2})=0_{R_{2}}$.
Since $f$ is one to one, then $r_{1}^{\prime}=0_{R_{1}}$ and so $(r_{1}%
^{\prime},f(r_{1}^{\prime})+j^{\prime})=0_{R_{1}\Join^{f}J}$ as needed. By
assumption, $(r_{1},r_{2}))(m_{1},m_{2})\in\overline{N_{2}}^{\varphi}$ and so
$r_{2}m_{2}\in N_{2}$. Therefore, $N_{2}$ is a semi $n$-submodule of $M_{2}$.

(2) Let $(r_{1},f(r_{1})+j)\in R_{1}\Join^{f}J$ and $(m_{1},\varphi(m_{1}))\in
M_{1}\Join^{\varphi}JM_{2}$ such that $(r_{1},f(r_{1})+j)^{2}(m_{1}%
,\varphi(m_{1}))\in\overline{N_{2}}^{\varphi}$, $(r_{1},f(r_{1})+j)\notin
\sqrt{0_{R_{1}\Join^{f}J}}$ and $Ann_{R_{1}\Join^{f}J}((m_{1},\varphi
(m_{1})))=0_{R_{1}\Join^{f}J}$. Then $(f(r_{1})+j)^{2}\varphi(m_{1})\in N_{2}%
$. Suppose on contrary that $f(r_{1})+j\in\sqrt{0_{R_{2}}}$. Then $f(r_{1}%
)\in\sqrt{0_{R_{2}}}$ as $J\subseteq\sqrt{0_{R_{2}}}$. Since $f$ is one to
one, then $r_{1}\in\sqrt{0_{R_{1}}}$ and so $(r_{1},f(r_{1})+j)\in
\sqrt{0_{R_{1}\Join^{f}J}}$, a contradiction. Therefore, $f(r_{1}%
)+j\notin\sqrt{0_{R_{2}}}$. Moreover, we prove that $Ann_{R_{2}}(\varphi
(m_{1}))=0_{R_{2}}$. Suppose $r_{2}\varphi(m_{1})=0_{M_{2}}$ for
$r_{2}=f(r_{1})\in R_{2}$. Then $\varphi(r_{1}m_{1})=0_{M_{2}}$ and so
$r_{1}m_{1}=0_{M_{1}}$ as $\varphi$ is one to one. Thus, $(r_{1},r_{2}%
)(m_{1},\varphi(m_{1}))=0_{M_{1}\Join^{\varphi}JM_{2}}$ and by assumption,
$(r_{1},r_{2})=0_{R_{1}\Join^{f}J}$. It follows that $r_{2}=0_{R_{2}}$ and
$Ann_{R_{2}}(\varphi(m_{1}))=0_{R_{2}}$. Since $N_{2}$ is\ a semi
$n$-submodule of $M_{2}$, then $(f(r_{1})+j)\varphi(m_{1})\in N_{2}$ and so
$(r_{1},f(r_{1})+j)(m_{1},\varphi(m_{1}))\in\overline{N_{2}}^{\varphi}$.
\end{proof}

\begin{corollary}
\label{Dup2}Let $N$ be a submodule of an $R$-module $M$ and $J$ be an ideal of
$R$. Then
\end{corollary}

\begin{enumerate}
\item If $\bar{N}$ is an $n$-submodule of $M\Join J$, then $N$ is an
$n$-submodule of $M$. The converse is true if $JM=0_{M}$.

\item If $\bar{N}$ is a semi $n$-submodule of $M\Join J$, then $N$ is a semi
$n$-submodule of $M$. The converse is true if $J\subseteq\sqrt{0}\cap
Ann_{R}(M)$.
\end{enumerate}

\begin{proof}
The proof is similar to that of Corollary \ref{Dup1} and left to the reader.
\end{proof}

\medskip

\textbf{Conflicts of Interest}

The authors have NO affiliations with or involvement in any organization or
entity with any financial interest or non-financial interest in the subject
matter or materials discussed in this manuscript.

\end{document}